\documentclass[a4paper,11pt]{article}
\usepackage[utf8]{inputenc}
\usepackage{amsfonts}
\usepackage{amsthm}
\usepackage{amsmath}
\usepackage{setspace}
\usepackage{array}
\usepackage{setspace}
\usepackage[margin=3cm]{geometry}
\numberwithin{equation}{section}
\usepackage{longtable}
\newtheorem{theorem}{Theorem}[section]
\newtheorem{lemma}[theorem]{Lemma}
\newtheorem{cor}[theorem]{Corollary}
\newtheorem{prop}[theorem]{Proposition}
\newtheorem{hypothesis}[theorem]{Hypothesis}
\theoremstyle{definition}

\theoremstyle{remark}
\newtheorem{remark}[theorem]{Remark}

\numberwithin{equation}{section}



\title{Cohomogeneity-One Quasi-Einstein Metrics}
\author{Timothy Buttsworth}
\begin{document}
\maketitle
\begin{abstract}
Let $G/H$ be a connected, simply connected homogeneous space of a compact Lie group $G$. We study $G$-invariant quasi-Einstein metrics on the cohomogeneity one manifold 
$G/H\times (0,1)$ imposing the so-called monotypic condition on $G/H$. We obtain estimates on the rate of blow-up for these metrics near a singularity under a mild assumption on $G/H$. Next, we demonstrate that we can find quasi-Einstein metrics satisfying arbitrary $G$-invariant Dirichlet conditions.
\end{abstract}

\section{Introduction}
%(John Lott, Sturm)
On a smooth manifold $M$, choose some non-negative real number $m$, 
a smooth function $u$ and a Riemannian metric $g$. 
The $m$--Bakry-Emery tensor $Ric^m_u(g)$ is defined by 
\begin{align*}
 Ric^m_u(g)=Ric(g)+Hess(u)-m\,du\otimes du,
\end{align*}
where $Ric(g)$ is the Ricci curvature, and $Hess(u)$ denotes the Hessian of the function $u$ with respect to $g$.
The $m$--Bakry-Emery tensor can be thought of as an extension of the Ricci curvature because the two notions are identical when the function $u$ is constant. 
When $m=0$, the $m$--Bakry-Emery tensor coincides with the usual Bakry-Emery tensor. The case that $m$ is strictly positive has been studied, for example, in
\cite{Case,Wei,KimKim,WeiWyley}. One setting in which this case arises is the study of the smooth metric measure space $(M,g,e^{-u}dV(g))$, 
where $dV(g)$ is the volume form of the Riemannian metric $g$; for example, see \cite{WeiWyley}. 
  
Motivated by the extensive theory of Einstein metrics, one would like to develop a theory for solutions of
\begin{align}\label{QEE}
 Ric^m_u(g)=\lambda g.
\end{align}
In particular, one would like to know under what circumstances solutions exist, and how they behave. 
A pair $(g,u)$ solving \eqref{QEE} is called a \textit{quasi-Einstein metric}. 
Clearly, a quasi-Einstein metric corresponds to an Einstein metric if $u$ is constant, but there is a relationship between 
the two concepts even if $u$ is non-constant. Indeed, as discussed in 
\cite{KimKim}, Einstein metrics on warped product spaces arise as solutions to the quasi-Einstein equation \eqref{QEE}. 
This observation was used by Case in \cite{Case} to demonstrate non-existence 
of Einstein metrics on warped product spaces under certain conditions. 
The relationship between Einstein and quasi-Einstein metrics also appears in \cite{Wei}, where the authors prove rigidity results
for \eqref{QEE}. We discuss this relationship further
in Section 2. 

In this paper, we consider solutions of \eqref{QEE} in the cohomogeneity one setting, obtaining an estimate on the blow-up rate and an existence theorem. 
Accordingly, we assume that $M$ is acted on by a group $G$ with principal orbits of codimension one, 
and we require that $g$ and $u$ be invariant under the action of $G$. The benefit of studying quasi-Einstein metrics in this setting is that \eqref{QEE} 
becomes a system of 
ODEs instead of PDEs. 
A similar setting is used by Hall in \cite{HallQE}, 
where $M$ appears as a one-parameter family of hypersurfaces. In this setting, Hall provides 
examples of quasi-Einstein metrics that respect this hypersurface structure, thus extending 
work on the Ricci soliton equations done by Dancer and Wang in \cite{DancerWangS}. Our focus is on the case that 
$M=G/H\times (0,1)$ where $G/H$ is a homogeneous space satisfying the monotypic condition, in which case \eqref{QEE} becomes a system of ODEs on the interval $(0,1)$. 
We study this system in two different contexts.

Firstly, we examine singular solutions for \eqref{QEE}, by which 
we mean solutions of the system of ODEs that exist on $(0,1)$, but cannot be smoothly extended to $[0,1)$. In terms of metrics, this means that a quasi-Einstein metric 
exists on $G/H\times (0,1)$, 
and cannot be extended smoothly to a metric on the manifold $G/H\times [0,1)$. 
Despite the name, there are singular solutions such that the completion of $G/H\times (0,1)$ is a smooth Riemannian manifold with respect to the corresponding Riemannian metric. 
These are solutions satisfying so-called `smoothness' conditions, and have been studied for cohomogeneity one Einstein metrics, 
Ricci solitons and quasi-Einstein metrics in \cite{EschenburgWang}, \cite{Buzano11} and \cite{Wink} respectively. Geometrically, 
these smoothness conditions correspond to Riemannian metrics on $G/H\times (0,1)$ such that the principal orbits $G/H\times \{t\}$ collapse smoothly as $t$ tends to $0$. 
In this paper, we estimate the blow-up rate of 
the singularities that can eventuate. In particular, we demonstrate under a mild assumption on the homogeneous space $G/H$, that the singularities can form no faster than in the situation
discussed in \cite{Wink}. Since Einstein metrics are examples of quasi-Einstein metrics, we also find that singularities of the Einstein equation 
form no faster than those in \cite{EschenburgWang}.

Next, we examine \eqref{QEE} subject to $G$-invariant Dirichlet conditions. In this case, we search for 
a quasi-Einstein metric $(g,u)$ on $G/H\times (0,1)$ which can be smoothly extended to $G/H\times [0,1]$ so that 
for $i=0,1$, $(g,u)$ coincides with $(\hat{g}_i,u(i))$ when restricted to 
$G/H\times \{i\}$, where $\hat{g}_i$ is a fixed $G$-invariant Riemannian metric on $G/H$, and $u(i)$ is a fixed real number. 
The Dirichlet problem for Einstein metrics has been 
studied in \cite{Anderson08,Buttsworth18}, but various other boundary-value problems for equations involving the Ricci curvature have also been 
studied by a number of authors; 
for example, see \cite{Anderson08,Buttsworth18,Shen,PulemotovRF,APC,APQLRF,Brendle2a,Brendle2b,PG,Cort,G16,CA12A}. 

 Despite the relationship between quasi-Einstein metrics and Einstein metrics on warped product spaces, 
 being able to solve the Dirichlet problem for Einstein metrics does not immediately imply the solvability of the Dirichlet problem for quasi-Einstein metrics; 
 see Section 2 for more details. In particular, the study of the Dirichlet problem for cohomogeneity one Einstein metrics in \cite{Buttsworth18} is 
 not immediately helpful 
 in the study of the Dirichlet problem for cohomogeneity one quasi-Einstein metrics. 
 Our result proves existence of a one-parameter family of solutions to the Dirichlet problem $(g,u)_{h}$, 
 where the parameter $h$ is small and coincides with the length of the interval $[0,1]$ with respect to the corresponding metric. 
 We also provide examples demonstrating that we may lose existence or uniqueness if $h$ is sufficiently large. 
\section{Preliminaries and Main Results}

\subsection{Cohomogeneity One Riemannian Metrics}
Let $G/H$ be a connected, simply connected homogeneous space of a compact Lie group $G$. 
Denote with $\mathfrak{g}$ and $\mathfrak{h}$ the Lie algebras of $G$ and $H$ respectively. We choose some Ad$(G)$-invariant inner product $Q$ on $\mathfrak{g}$. 
Letting $\mathfrak{m}$ be the $Q$-orthogonal complement 
of $\mathfrak{h}$ in $\mathfrak{g}$, we naturally identify $\mathfrak{m}$ with the tangent space of $G/H$ at $H$. Now take a $Q$-orthogonal 
decomposition  
\begin{equation}\label{DC}
\mathfrak{m}=\bigoplus_{i=1}^{n} \mathfrak{m}_i
\end{equation}
such that each $\mathfrak{m}_i$ in \eqref{DC} is an irreducible Ad$(H)$ module.
We make the following assumption on this decomposition. 
\begin{hypothesis}\label{IID}
 The submodule $\mathfrak{m}_{i}$ is non-isomorphic to $\mathfrak{m}_{j}$ if $i\neq j$. 
\end{hypothesis}
\noindent This assumption is referred to as the \textit{monotypic} condition, and it has appeared in, for example, \cite{DancerWang}, \cite{GroveZiller} and \cite{Graev}.

We search for solutions $(g,u)$ 
of \eqref{QEE} such that $u$ is a $G$-invariant function and $g$ is a $G$-invariant Riemannian metric on $G/H\times (0,1)$ having the form 
\begin{equation}\label{RMG}
g=h^2(t)\,dt\otimes dt+g_t,
\end{equation}
where $t$ is the natural parameter running through the interval $(0,1)$, 
$h(t)$ is a smooth positive function on $(0,1)$, and 
$g_t$ is a time-dependent $G$-invariant Riemannian metric on $G/H$. Hypothesis \ref{IID} implies the existence of an array of smooth functions 
$y=(y_1,\cdots,y_n)$ on $(0,1)$ such that 
\begin{align*}
g_t(X,Y)&=\sum_{i=1}^{n} e^{2y_i(t)} Q(pr_{\mathfrak{m}_i} X,pr_{\mathfrak{m}_i}Y)
\end{align*}
for all $X,Y\in \mathfrak{m}$, where $pr_{\mathfrak{m}_i} X$ denotes the $Q$-orthogonal projection of $X$ in $\mathfrak{m}_i$. 
The quasi-Einstein equation \eqref{QEE} is diffeomorphism-invariant, so by re-parametrising, we can assume that the 
function $h$ is constant. In this case, a tedious but straightforward computation of $Ric(g)$ (cf. Proposition 1.14 of \cite{GroveZiller} and Lemma 3.1 of \cite{APC}) and $Hess(u)$  demonstrates that the pair $(g,u)$ satisfies \eqref{QEE} if and only if 
\begin{align}\label{QEECO}
\begin{split}
 -\sum_{k=1}^{n}d_k(y_k''+y_k'^2)+u''-m(u')^2&=h^2\lambda, \\
 h^2r_i(y)-y_i'\sum_{k=1}^{n}d_ky_k'+u'y_i'-y_i''&=h^2\lambda, \qquad \ i=1,\cdots,n,
\end{split}
 \end{align}
 where $r_i(y)$ is given by 
 \begin{align}\label{DHR}
  r_i(y)=\frac{\beta_i}{2e^{2y_i}}+\sum_{k,l=1}^{n}\gamma^{l}_{ik}\frac{e^{4y_i}-2e^{4y_k}}{4 e^{2y_i+2y_k+2y_l}}.
 \end{align}
 Here, $d_i$ is the dimension of $\mathfrak{m}_i$, and $\beta_i$ and $\gamma_{ik}^{l}$ are non-negative numbers determined by the homogeneous space $G/H$; 
 see Section 3 of \cite{APC} for the precise definition of these numbers, but note they appeared earlier in, for example, \cite{Wang}, with different notation.  

 Let $r$ be the continuous vector function with components $r_i$, set 
 $$\xi=\sum_{i=1}^{n}d_i y_i'-u',$$ and identify $y$ with the diagonal matrix 
 \begin{align*}
diag(\underbrace{y_1,\cdots,y_1}_{d_1~\mathrm{times}},\cdots,\underbrace{y_n,\cdots,y_n}_{d_n~\mathrm{times}})
\end{align*} 
of dimension $d=\sum_{i=1}^{n}d_i$. Then we can simplify \eqref{QEECO} further as 
 \begin{align}\label{QEECOS}
 \begin{split}
 \xi'&=-tr(y'^2)-m(tr(y')-\xi)^2-h^2\lambda, \\
 y''&=-\xi y'+h^2r(y)-h^2\lambda I, 
\end{split}
 \end{align}
where $I$ is the identity matrix.
 \subsection{Relationship between quasi-Einstein and Einstein metrics}
We now briefly discuss the relationship between quasi-Einstein and Einstein metrics 
in the cohomogeneity one setting. The following is a cohomogeneity one version of Proposition 5 of \cite{KimKim}. 
\begin{prop}\label{PQETE}
 Suppose that $d_0=\frac{1}{m}$ is a positive integer and $(g,u)$ is a $G$-invariant quasi-Einstein metric on $G/H\times (0,1)$ with constant $\lambda$. 
 Then there exists $\mu>0$ such that if $(F,k)$ is a $d_0$-dimensional homogeneous Einstein manifold with Einstein constant $\mu$, then 
 $g\oplus e^{-2mu}k$ is a cohomogeneity one Einstein metric on $G/H\times (0,1)\times F$ with Einstein constant $\lambda$. 
\end{prop}
Proposition \ref{PQETE} allows us to use results about cohomogeneity one Einstein metrics to find cohomogeneity one quasi-Einstein metrics. For example, 
suppose $G/H$ satisfies Hypothesis \ref{IID} and we aim to find quasi-Einstein metrics with $g$ having the form \eqref{RMG} with $h^2=1$, and we prescribe 
 $y_i(0),y_i'(0),u(0)$ and $u'(0)$. For this initial-value problem, \eqref{QEECO} gives us $u''(0)$ explicitly, and
Lemma 5 of \cite{HallQE} implies that 
\begin{align*}
  \mu&=v v''+vv' \sum_{i=1}^{n}d_i y_i'+\left(\frac{1}{m}-1\right)v'^2+\lambda v^2,
\end{align*}
where $v=e^{-mu}$. Therefore, the $\mu$ in Proposition \ref{PQETE} can be found explicitly. 
The quasi-Einstein metric can then be recovered by solving the corresponding initial-value problem for the cohomogeneity one 
Einstein metric $g\oplus e^{-2mu}k$ on $G/H\times (0,1)\times F$, where $(F,k)$ has Einstein constant $\mu$. 

This method is less useful for the Dirichlet problem because we cannot know what $\mu$ is from merely the Dirichlet conditions without solving the 
entire equation. In particular, Theorem 2.2 of \cite{Buttsworth18}, which solves the Dirichlet problem for cohomogeneity one Einstein metrics,
is not directly applicable to the search for solutions of the Dirichlet problem for quasi-Einstein metrics. 

%To see this relationship, suppose that we have 
%a cohomogeneity-one quasi-Einstein metric $(g,u)$ with $m>0$. The $y_i$ and $u$ functions associated with this metric solve \eqref{QEECO}, so by letting $y_0=-mu$ and $d_0=\frac{1}{m}$, we see that the collection of smooth functions $\{y_i\}_{i=0}^{n}$ solve 
%\begin{align}\label{QEECO1}
%\begin{split}
% -\sum_{k=0}^{n}d_k(y_k''+y_k'^2)&=p_2\lambda, \\
% p_2r_i(y)-y_i'\sum_{k=0}^{n}d_ky_k'-y_i''&=p_2\lambda.
%\end{split}
%\end{align} Corollary 3 of \cite{KimKim} (cf. Lemma 2.2 of \cite{HallQE}) implies that for this quasi-Einstein metric, there exists a constant $\mu$ such that  \begin{align}\label{CLQEM}
% \mu e^{-2y_0}-y_0'\sum_{k=0}^{n}d_k y_k'-y_0''&=p_2\lambda.
%\end{align}
%For a given $\mu$, if $\frac{1}{m}$ is an integer, then \eqref{QEECO1} and \eqref{CLQEM} form a system of equations having the form of the cohomogeneity one Einstein 
%equations. 
%This means that we can use results for the initial value problems for Einstein metrics to solve the initial value problem for quasi-Einstein metrics. Indeed, if 
%we aim to solve \eqref{QEECO1} with $\lambda$, $y_i(0)$ and $y_i'(0)$ fixed for $i=0,\cdots,n$, then \eqref{QEECO1} tells us what $y_0''(0)$ is, meaning 
%we can explicitly find the $\mu$ in \eqref{CLQEM}. The system \eqref{QEECO1} and \eqref{CLQEM} can then be treated as an initial value problem for a 
%cohomogeneity one Einstein metric. 

\subsection{Singular Solutions}
 The first goal of this paper is to study the singularities 
 of cohomogeneity one solutions to \eqref{QEE}. By singularity, we mean a solution of \eqref{QEECOS} on the interval $(0,1)$ which develops a singularity 
 at $t$ is sent to $0$.
 In \cite{EschenburgWang}, \cite{Buzano11} and \cite{Wink}, 
 solutions with singularities are constructed to the cohomogeneity one Einstein, Ricci soliton and quasi-Einstein equations respectively. 
 In this section, instead of dealing with the existence of such singular 
 solutions, we assume such a solution exists, and examine how quickly the solution becomes singular. 
 We show that the singularities form no quicker than those studied in  \cite{EschenburgWang}, \cite{Buzano11} and \cite{Wink} (see Remark \ref{DOQEEMB} for more details). 
 
 To make this more precise, assume we have a solution $(y,\xi)=(y_1,\cdots,y_n,\xi)$ of \eqref{QEECOS} for some $h>0$ and $\lambda\in \mathbb{R}$. We introduce 
 the function $R:\mathbb{R}^n\to \mathbb{R}^+$ with \begin{equation}\label{DOrR}
                                                     R(x)=\sum_{i=1}^{n}\beta_i e^{-2x_i}+\sum_{i,j,k=1}^{n}\gamma_{ij}^{k}e^{2x_i-2x_j-2x_k}
                                                    \end{equation}
                                                    as well as the diagonal matrix $L(t)=y'(t)$, and define $M(t)\ge 0$ so that 
\begin{align*}
M(t)^2&=\xi(t)^2+tr(L(t)^2)+R(y(t))\\
&=\left(\sum_{i=1}^{n}d_i y_i'(t)-u'(t)\right)^2+\sum_{i=1}^{n}d_i y_i'(t)^2+R(y(t)). 
\end{align*}
Recalling that $r$ is given by \eqref{DHR}, we see that $r$ is bounded by $R$, so if $M(t)$ is bounded on $(0,1)$, then the solution $(y,\xi)$ of 
\eqref{QEECOS} is also bounded in 
$C^2((0,1);\mathbb{R}^n)\times C^1((0,1);\mathbb{R})$. By the standard theory of ODEs, solutions of \eqref{QEECOS} can be extended as long as they are bounded, 
so any solution which is singular around $t=0$ satisfies  $\sup_{t> 0}M(t)=\infty$. In this sense we see that $M(t)$ 
controls the growth of a solution to \eqref{QEECOS}; a singularity 
occurs if and only if $M(t)$ is unbounded. 

For convenience, we now choose some 
$0<T<1$ and consider our solution of \eqref{QEECOS} on the half-closed interval $(0,T]$. 
We make one additional hypothesis on the homogeneous space $G/H$, and then state our main result. 
 \begin{hypothesis}\label{HN2}
For each $i$, the dimension of $\mathfrak{m}_i$ is at least $2$. 
\end{hypothesis}
\begin{remark}
There is a large variety of homogeneous spaces satisfying Hypothesis \ref{HN2}. For instance, 
all generalised flag manifolds satisfy Hypothesis \ref{HN2}. For the definition of a generalised flag manifold, see Chapter 7 of 
\cite{ArvFM}.
\end{remark}
 \begin{theorem}\label{MTFS}
 Fix $\lambda\in \mathbb{R}$, $h>0$ and $m>0$, and let $(g,u)$ be a singular cohomogeneity one quasi-Einstein metric on $G/H\times (0,T]$, where 
 $G/H$ is a homogeneous space satisfying Hypotheses \ref{IID} and \ref{HN2}. 
 Then for the associated pair $(y,\xi)$ solving \eqref{QEECOS} on $(0,T]$, 
 we have $\sup_{t> 0}M(t)t<\infty$. 
 \end{theorem}
 The proof of this result is provided in Section 3. 
 \begin{remark}\label{DOQEEMB}
  In \cite{EschenburgWang}, 
  cohomogeneity one Einstein metrics are found subject to smoothness conditions. 
  These solutions consist of a family of Riemannian metrics $g(t)$ on a homogeneous space $G/H$, 
  with $L(t)$ denoting the associated shape operator, and $r(t)$ denoting the Ricci curvature of the homogeneous space. Here, $g(t)$, $L(t)$ and $r(t)$ are treated 
  as three families of linear operators on $\mathfrak{m}$, the tangent space of $G/H$.
  A solution of the Einstein equation is found so that $g$, $L$ and $r$ all split into regular and singular parts. The singular parts of $g,L$ and $r$ behave like $t^2,\frac{1}{t}$ and $\frac{1}{t^2}$ respectively. In our context, $L$ and $r$ have the same role as they do in \cite{EschenburgWang}, 
  so the Einstein metrics found in \cite{EschenburgWang} satisfy $\sup_{t> 0}M(t)t<\infty$ if we take $u=0$. 
  For cohomogeneity one Ricci solitons, the additional constraint that $u'(0)=0$ is imposed in \cite{Buzano11}, 
  in which case $t\xi=(tr(L)-u')t$
  is also bounded, so we again have  $\sup_{t> 0}M(t)t<\infty$. Similar conditions are imposed in \cite{Wink}. Thus, Theorem \ref{MTFS} demonstrates that, under certain conditions,
  singularities 
  of the cohomogneity one quasi-Einstein equations form no faster than the singularities already studied for certain cohomogeneity one geometric equations. 
 \end{remark}

 \subsection{Dirichlet Conditions}
The second goal of this paper is to study the problem of solving \eqref{QEE} subject to Dirichlet conditions for $g$ and $u$. 
\begin{theorem}\label{DCT}
 Fix $\lambda\in \mathbb{R}$, two $G$-invariant Riemannian metrics $\hat{g}_0,\hat{g}_1$ on a homogeneous space $G/H$ satisfying Hypothesis \ref{IID}, and two numbers $u_0,u_1$. 
 Then there exists a one-parameter 
 family of quasi-Einstein metrics $(g,u)$ on $G/H\times (0,1)$ with constant $\lambda$, and $g$ having the form \eqref{RMG} such that for $i=0,1$,
 $g_i=\hat{g}_i$ and 
 $u(i)=u_i$. 
\end{theorem}
\begin{remark}
 For Theorem \ref{DCT}, we do not need to assume that $G/H$ is simply connected. 
\end{remark}

Let $\hat{g}_0,\hat{g}_1$ be two $G$-invariant Riemannian metrics on the homogeneous space $G/H$. 
Since we impose Hypothesis \ref{IID},
there exists two arrays of real numbers $\{a_i\}_{i=1}^{n}$ and $\{b_i\}_{i=1}^{n}$ such that for all $X,Y\in \mathfrak{m}$, 
\begin{align*}
 \hat{g}_0(X,Y)=\sum_{i=1}^{n}e^{2a_i}Q(pr_{\mathfrak{m}_i}X,pr_{\mathfrak{m}_i}Y),\\
 \hat{g}_1(X,Y)=\sum_{i=1}^{n}e^{2b_i}Q(pr_{\mathfrak{m}_i}X,pr_{\mathfrak{m}_i}Y).\\
\end{align*}
If we assume our metric $g$ has the form of \eqref{RMG}, to prove Theorem \ref{DCT}, it suffices to find a one-parameter family of solutions of \eqref{QEECO} subject to the boundary conditions
\begin{align}\label{QEEWP0}
\begin{split}
 y_i(0)=a_i, &\qquad y_i(1)=b_i,\\
 u(0)=u_0, &\qquad  u(1)=u_1. 
 \end{split}
\end{align}
If we define $c=\sum_{i=1}^{n}d_i(b_i-a_i)-(u_1-u_0)$, solving \eqref{QEECO} subject to \eqref{QEEWP0} is equivalent to solving \eqref{QEECOS} subject to the conditions 
\begin{align}\label{QEEWP}
\begin{split}
 y_i(0)=a_i, &\qquad y_i(1)=b_i,\\
 \int_{0}^{1}\xi &=c.
 \end{split}
\end{align} 
In Section 4, we find a $K>0$ such that a solution $(y,\xi)$ of \eqref{QEECOS} subject to \eqref{QEEWP} exists for each $0<h^2<K$, so we can take $h$ to be our parameter in Theorem \ref{DCT}. 
In Section 5, we demonstrate that we cannot expect existence and uniqueness without $h^2$ being small.
\section{Singularity Analysis}
In this section, we prove Theorem \ref{MTFS}. 
To do this, we first assume that we have a solution $(y,\xi)$ of \eqref{QEECOS} defined on the interval $(0,T]$, 
and suppose that the solution becomes unbounded around $t=0$. We can assume without loss of generality that $h^2=1$, since changing $h^2$ simply 
has the effect of changing the 
interval $(0,T]$ to some other interval $(0,T^*]$. 
The properties of $r$ in \eqref{DHR} are important to this analysis, and the relevant facts about $r$ are given in the following lemma. 
\begin{lemma}\label{DOR}
Let $r$ be the function with components given by \eqref{DHR} and $R$ as in \eqref{DOrR}. Then:

(i) $\sup_{y\in \mathbb{R}^n}\frac{\left|r(y)\right|}{R(y)}=:c_1<\infty$

(ii) $\sup_{y\in \mathbb{R}^n}\frac{\left|Dr(y)\right|}{R(y)}=:c_2<\infty$, where $Dr$ is the Jacobian matrix of $r$

(iii) $\inf_{y\in \mathbb{R}^n} \frac{\left|r(y)\right|}{R(y)}=:c_3>0$, if Hypothesis \ref{HN2} is satisfied. 
\end{lemma}
\begin{proof}
Statements (i) and (ii) are straightforward to verify since $r$ consists of sums and differences of exponential terms, all of which appear summed in $R$. 

To verify (iii), we note that it suffices to demonstrate that 
\begin{align}\label{EFCP}
 \inf_{R(y)=1} \frac{\left|r(y)\right|}{R(y)}= \inf_{R(y)=1} \left|r(y)\right|>0.
\end{align}
This is because of the scaling properties of 
$r$ and $R$. 
To establish \eqref{EFCP}, we make extensive use of results of \cite{Graev}. 
Firstly, $\mu=\nabla R$ is a diffeomorphism from $R^{-1}(1)$ to $\Delta$, 
where $\Delta\subset \mathbb{R}^n$ is a convex polytope depending on the homogeneous space $G/H$. This is Theorem 1 of \cite{Graev}, but note that the result is essentially a consequence of 
an un-numbered lemma in Section 4.2 of \cite{Fulton}.  
With this diffeomorphism, we can now treat $r$ as a function from $\Delta$ to $\mathbb{R}^{n}$, 
and it suffices to demonstrate that $\inf _{x\in\Delta} \left|r(x)\right|>0$. 

Sections 4 and 5 of \cite{Graev} explain how $r$ can be extended continuously to $\Gamma=\partial \Delta$. 
This is done using contractions of Lie algebras. 
Now if $r(x)=0$ for some $x\in \Delta$, then Bochner's Theorem (see, for example, Theorem 1.84 of \cite{Besse}) implies that the isometry group of $G/H$ with respect to the metric $x$ is abelian. 
From this we conclude that $[\mathfrak{m},\mathfrak{h}]=0$ (cf. Section 2.7 of \cite{APMG}), so $\mathfrak{m}$ 
is completely reducible into one-dimensional $Ad(H)$-invariant modules, which contradicts Hypothesis \ref{HN2}. 
Furthermore, if $x\in \Gamma$ and $r(x)=0$, Theorem 3 and Lemma 3 of \cite{Graev} together
imply the existence of a proper subset $I\subset \{1,\cdots,n\}$ such that $[\mathfrak{m}_i,\mathfrak{h}]=0$ for each $i\in I$. We arrive at another contradiction 
with Hypothesis \ref{HN2}. 
Therefore, the continuous function $r$ does not achieve a value of $0$ on the compact set $\Delta\cup \Gamma$, so $\inf _{x\in \Delta} \left|r(x)\right|>0$ as required. 
\end{proof}
%\begin{remark}
%Lemma 3 of \cite{Graev} implies that $\mathfrak{m}_i\cap \mathfrak{z}(\mathfrak{g})=0$, %where $\mathfrak{z}(\mathfrak{g})$ is the centre of $\mathfrak{g}$. 
%Therefore, instead of Hypothesis \ref{HN2}, we could require that any $1$-dimensional module $\mathfrak{m}_i$ is central in $\mathfrak{g}$, and still reach conclusion (iii) of Lemma \ref{DOR}.  
%\end{remark}
Suppose we have a solution of \eqref{QEECOS} defined on $(0,T]$ such that 
$M(t)t$ becomes unbounded around $0$. 
Choose some sequence $\{T^{(k)}\}_{k=1}^{\infty}\subset (0,T]$ such that $T^{(k)}\to 0$ and $M(T^{(k)})T^{(k)}\to \infty$, and choose $t^{(k)}\ge T^{(k)}$ such that 
\begin{align*}
M(t^{(k)})(t^{(k)}-T^{(k)})= \sup_{t\in [T^{(k)},T]} M(t)(t-T^{(k)})\ge 0.
\end{align*}
Consider the rescaled functions 
\begin{align}\label{DRSF}
\begin{split}
y^{(k)}(t)&=y\left(t^{(k)}+\frac{t}{M(t^{(k)})}\right),\\
L^{(k)}(t)&=\frac{1}{M(t^{(k)})}L\left(t^{(k)}+\frac{t}{M(t^{(k)})}\right),\\
\xi^{(k)}(t)&=\frac{1}{M(t^{(k)})}\xi\left(t^{(k)}+\frac{t}{M(t^{(k)})}\right),\\
\end{split}
\end{align}
which can be defined for $t\in [M(t^{(k)})(T^{(k)}-t^{(k)}), M(t^{(k)})(T-t^{(k)})]$. 
These functions solve 
\begin{align}\label{EFRF}
\begin{split}
 (\xi^{(k)})'&=-tr((L^{(k)})^2)-m(tr(L^{(k)})-\xi^{(k)})^2-\frac{\lambda}{M(t^{(k)})^2}, \\
 (L^{(k)})'&=-\xi^{(k)} L^{(k)}+\frac{r(y^{(k)})-\lambda I}{M(t^{(k)})^2},\\ 
 (y^{(k)})'&=L^{(k)}.
\end{split}
\end{align}
We see that 
\begin{align*}
\sup_{k} M(t^{(k)})(t^{(k)}-T^{(k)})&=\sup_{k}\sup_{t\in(0,T]} M(t)(t-T^{(k)})\\
&=\sup_{t\in (0,T]}\sup_{k}M(t)(t-T^{(k)})\\
&=\sup_{t\in (0,T]}M(t)t \\
&=\infty,
\end{align*}
which implies that, by taking a subsequence in $k$ if necessary, $M(t^{(k)})(T^{(k)}-t^{(k)})\to -\infty$. This implies that $M(t^{(k)})\to \infty$, so we must also have $t^{(k)}\to 0$ (since $M(t)$ is bounded away from $0$) and $M(t^{(k)})(T-t^{(k)})\to \infty$. Therefore, the rescaled functions of \eqref{EFRF} are defined for increasingly large intervals. 
The following lemma demonstrates that these rescaled functions are bounded on any fixed large interval.
\begin{lemma}\label{NRIB}
For all $a>0$, there exists $N\in \mathbb{N}$ such that for $k\ge N$:

(i) $tr(L^{(k)}(t)^2)+\xi^{(k)}(t)^2+\frac{\left|r(y^{(k)}(t))-\lambda I\right|}{M^2(t^{(k)})}\le c_1+2$ for all $t\in [-a,a]$ 

(ii) the derivative of $\frac{r(y(t))-\lambda}{M^2(t^{(k)})}$ in $t$ is bounded in $[-a,a]$ by $c_2\sqrt{c_1+2}$.
\begin{flushleft}
Here, $c_1$ and $c_2$ are the constants defined in Lemma \ref{DOR}. 
\end{flushleft}
\end{lemma}
\begin{proof}
We see that
 \begin{align*}
 &tr(L^{(k)}(t)^2)+\xi^{(k)}(t)^2+\frac{\left|r(y^{(k)}(t))\right|}{M^2(t^{(k)})}\\
 &\le tr(L^{(k)}(t)^2)+\xi^{(k)}(t)^2+c_1\frac{R(y^{(k)}(t))}{M^2(t^{(k)})}\\
  &=\frac{tr\left(L(t^{(k)}+\frac{t}{M(t^{(k)})})^2\right)+\xi\left(t^{(k)}+\frac{t}{M(t^{(k)})}\right)^2+c_1R\left(y(t^{(k)}+\frac{t}{M(t^{(k)})})\right)}{M^2(t^{(k)})}\\
  &\le \max\{1,c_1\} \frac{M^2(t^{(k)}+\frac{t}{M(t^{(k)})})}{M^2(t^{(k)})}\\
  &\le (c_1+1)\frac{(t^{(k)}-T^{(k)})^2}{(t^{(k)}+\frac{t}{M(t^{(k)})}-T^{(k)})^2}.
 \end{align*}
Now as $k$ is sent to $\infty$, $\frac{(t^{(k)}-T^{(k)})^2}{(t^{(k)}+\frac{t}{M(t^{(k)})}-T^{(k)})^2}$ converges to $1$ uniformly for $t\in [-a,a]$, 
so $$(c_1+1)\frac{(t^{(k)}-T^{(k)})^2}{(t^{(k)}+\frac{t}{M(t^{(k)})}-T^{(k)})^2}+\frac{\lambda}{M^2(t^{(k)})}\le c_1+2$$ for large $k$. 
This verfies assertion (i).

For assertion (ii), 
note that $$\frac{d\left(\frac{r(y^{(k)}(t))-\lambda}{M^2(t^{(k)})}\right)}{dt}=\frac{Dr(y^{(k)}(t))L^{(k)}(t)}{M^2(t^{(k)})}.$$ By (ii) of Lemma \ref{DOR}, 
$\left|Dr\right|(y^{(k)}(t))\le c_2R(y^{(k)}(t))$, so $\frac{\left|Dr(y^{(k)}(t))L^{(k)}(t)\right|}{M^2(t^{(k)})}\le c_2\left|L^{(k)}(t)\right|\le c_2 \sqrt{c_1+2}$ for large $k$. 
\end{proof}
\begin{cor}\label{AALA}
 For any $a>0$, there exists a subsequence of $(\xi^{(k)},L^{(k)},y^{(k)})_{k=1}^{\infty}$ such that 
 $L^{(k)}\to \tilde{L}$, $\xi^{(k)}\to \tilde{\xi}$ and $\frac{r(y^{(k)})-\lambda I}{M^2(t^{(k)})}\to \tilde{r}$ in $C^0[-a,a]$ as 
 $k\to \infty$, where $\tilde{L}$, $\tilde{\xi}$ and $\tilde{r}$ solve 
 \begin{align}\label{LEQE}
\begin{split}
\tilde{\xi}'&=-tr(\tilde{L}^2)-m(tr(\tilde{L})-\tilde{\xi})^2,\\
\tilde{L}'&=-\tilde{\xi}\tilde{L}+\tilde{r}.
\end{split}
\end{align}
\end{cor}
\begin{proof}
Lemma \ref{NRIB} implies that $L^{(k)}$, $\xi^{(k)}$ and $\frac{r(y^{(k)})-\lambda I}{M^2(t^{(k)})}$ are bounded in $C^0[-a,a]$. 
Equation \eqref{EFRF} then implies that $L^{(k)}$ and $\xi^{(k)}$ are also bounded in $C^1[-a,a]$, and we already know from Lemma \ref{NRIB} 
that $\frac{r(y^{(k)})-\lambda I}{M^2(t^{(k)})}$ is bounded in $C^1[-a,a]$ as well. 
The Arzela-Ascoli Theorem implies, by taking another subsequence if necessary,
that $L^{(k)}\to \tilde{L}$, $\xi^{(k)}\to \tilde{\xi}$ and $\frac{r(y^{(k)})-\lambda I}{M^2(t^{(k)})}\to \tilde{r}$ in $C^0[-a,a]$. By taking limits of 
the integral form of \eqref{EFRF}, we find that $\tilde{L}$, $\tilde{\xi}$ and $\tilde{r}$ solve \eqref{LEQE}. 
\end{proof}

\begin{proof}[Proof of Theorem \ref{MTFS}]
If $\sup_{t>0}M(t)t=\infty$, construct the rescaled functions in \eqref{DRSF}. 
We claim that $\liminf _{k\to \infty} \left|\xi^{(k)}(-1)\right|+\left|\xi^{(k)}(1)\right|=:2K>0$. To see this, we argue by contradiction. 
If $K=0$, Corollary \ref{AALA} implies the existence of a solution to \eqref{LEQE} on $[-1,1]$ with $\tilde{\xi}(-1)=\tilde{\xi}(1)=0$. 
It follows from \eqref{LEQE} that $\tilde{\xi}$ is monotone decreasing, so in fact $\tilde{\xi}=0$ on $[-1,1]$. We also find from \eqref{LEQE} that 
$\tilde{L}=\tilde{r}=0$ on $[-1,1]$. 
In particular, 
\begin{align*}
0&=tr(\tilde{L}(0)^2)+\tilde{\xi}(0)^2+\left|\tilde{r}(0)\right|\\
&=\lim_{k\to \infty} \frac{tr(L(t^{(k)})^2)+\xi(t^{(k)})^2+\left|r(y(t^{(k)}))\right|}{tr(L(t^{(k)})^2)+\xi(t^{(k)})^2+R(y(t^{(k)}))}.
\end{align*} We arrive at a contradiction with (iii) 
of Lemma \ref{DOR}. 

Now for an arbitrary $a>1$, Corollary \ref{AALA} again implies the existence of a solution of \eqref{LEQE} on $[-a,a]$, and we now know that 
$\left|\tilde{\xi}(-1)\right|\ge K$ or $\left|\tilde{\xi}(1)\right|\ge K$. 
The monotonicity of $\tilde{\xi}$ implies that $\tilde{\xi}(-1)\ge K$ or $\tilde{\xi}(1)\le -K$. Now since $m>0$,
$$\inf_{\tilde{\xi}\neq 0,\tilde{L}}\frac{tr(\tilde{L}^2)+m(tr(\tilde{L})-\tilde{\xi})^2}{\tilde{\xi}^2}=:q>0 .$$ Therefore, if $\tilde{\xi}(1)<-K$, then
 $$\tilde{\xi}(t)'\le -q \tilde{\xi}(t)^2\le -q K^2$$ for $t\ge 1$. 
For large $a$, we then find that $\tilde{\xi}(a)^2>(c_1+2)$. 
This is a contradiction with Lemma \ref{NRIB}. We arrive at a similar contradiction if $\tilde{\xi}(-1)\ge K$. 
\end{proof}

\section{The Dirichlet Problem}
In this section, we prove Theorem \ref{DCT}. Recall from Section 2.4 that it suffices to find a one-parameter family of solutions to 
\eqref{QEECOS} subject to \eqref{QEEWP}. Consider the problem of solving \eqref{QEECOS} subject to 
\begin{align}\label{QEEWP1}
\begin{split}
 y_i(0)=a_i, &\qquad y_i(1)=a_i+p(b_i-a_i),\\
 \int_{0}^{1}\xi &=pc,
 \end{split}
\end{align}
for some $p\in [0,1]$. By multiplying \eqref{QEECOS} by a Green's function, integrating and using~\eqref{QEEWP1}, we see that finding a solution $(y,\xi)$ of 
\eqref{QEECOS} and 
\eqref{QEEWP1} is equivalent to solving $(y,\xi)=H(p,h^2,y,\xi)$, 
where $$H:[0,1]\times \mathbb{R}^+\times C^1([0,1]; \mathbb{R}^n)\times C^0([0,1]; \mathbb{R})\to  C^1([0,1]; \mathbb{R}^n)\times C^0([0,1]; \mathbb{R})$$
is completely continuous (cf. Section 4 of \cite{Buttsworth18}). If $p=1$, \eqref{QEEWP1} coincides with \eqref{QEEWP}, so to prove Theorem \ref{DCT}, it suffices to 
find a solution of $H(1,h^2,y,\xi)=(y,\xi)$ for some small values of $h^2>0$. We do this by analysing the $h=0$ case, and then using Schauder degree theory to recover information about the $h^2>0$ case. 
\begin{remark}
Geometrically, setting $h^2=0$ has the effect of collapsing the length of the interval $[0,1]$ to $0$. If $h^2=0$, the equations of \eqref{QEEWP} coincide with the equation for steady ($\lambda=0$) 
quasi-Einstein metrics 
in the special case that the homogeneous space $G/H$ is the torus. 
\end{remark}
Using the diagonal matrix $L=y'$, we see that solving $H(p,0,y,\xi)=(y,\xi)$ is equivalent to 
\begin{align}\label{RSET}
\begin{split}
\xi'&=-tr(L^2)-m(tr(L)-\xi)^2,\\
L'&=-\xi L,
\end{split}
\end{align}
subject to 
\begin{align}\label{DCRST}
\begin{split}
\int_{0}^{1}L&=pD,\\
\int_{0}^{1}\xi&=p c,
\end{split}
\end{align}
with $D$ a diagonal matrix given by
\begin{align*}
D=diag(\underbrace{a_1-b_1,\cdots,a_1-b_1}_{d_1~\mathrm{times}},\cdots,\underbrace{a_n-b_n,\cdots,a_n-b_n}_{d_n~\mathrm{times}}).
\end{align*} The following lemma is a compactness result for solutions of \eqref{RSET} and \eqref{DCRST}, 
in that it demonstrates that a sequence of solutions 
has a convergent subsequence in a certain weak sense. 
\begin{lemma}\label{CBUS}
 Suppose that there is a sequence of solutions $(L^{(k)},\xi^{(k)},p^{(k)})$ to \eqref{RSET} and \eqref{DCRST} on $[0,1]$ with $p^{(k)}\in [0,1]$. 
 Then there exists a solution $(\bar{L},\bar{\xi})$ to \eqref{RSET} defined on 
 $(0,1)$ such that for each $0<a<\frac{1}{2}<b<1$, there exists a subsequence of $(L^{(k)},\xi^{(k)})$ converging to $(\bar{L},\bar{\xi})$ in $C^0[a,b]$. 
\end{lemma}
\begin{proof}
By taking a first subsequence of $(L^{(k)},\xi^{(k)})$, we can assume that for each $i$, $L_{i}^{(k)}(\frac{1}{2})$ 
and $\xi^{(k)}(\frac{1}{2})$ are monotone in $k$. 
 Fix $0<a<\frac{1}{2}<b<1$. 
 The proof relies on showing that $(L^{(k)},\xi^{(k)})$ is bounded in $C^0[a,b]$ independently of $k\in \mathbb{N}$. If we can do this, since $(L^{(k)},\xi^{(k)})$ 
 solves \eqref{RSET}, we also find that $(L^{(k)},\xi^{(k)})$ is bounded in $C^2[a,b]$. Therefore, the Arzela-Ascoli 
 Theorem implies that there is a subsequence of $(L^{(k)},\xi^{(k)})$ converging to $(\bar{L},\bar{\xi})$ in $C^1[a,b]$, so $(\bar{L},\bar{\xi})$ 
 is a solution of \eqref{RSET}. 
Since $L_{i}^{(k)}(\frac{1}{2})$ and $\xi^{(k)}(\frac{1}{2})$ are monotone in $k$, 
$\bar{L}(\frac{1}{2})$ and $\bar{\xi}(\frac{1}{2})$ are independent of $a$ and $b$, and by uniqueness of solutions to ODEs, $\bar{L}$ and $\bar{\xi}$
are themselves independent of $a$ and $b$, and can be extended to a solution on $(0,1)$ as required. 
We now drop reference to the superscript $k$ to simplify notation. 
 
 To find the required bounds on $(L,\xi)$, first suppose that there exists an $i$ such that $\sup_{k}\left|\left|L_{i}\right|\right|_{C^0[a,b]}=\infty$. 
Note that the second equation of \eqref{RSET} implies that 
\begin{align}\label{EFL2}
 (L_i^2)'=-2L_i^2\xi,
\end{align}
from which we can see that 
$L_{i}$ does not change sign on $[0,1]$. Also, from the first equation of \eqref{RSET}, we see that $\xi'\le 0$, 
so if $L_i$ is non-zero, it can have at most one critical point, which minimises $L_{i}^2$ by \eqref{EFL2}. 
Therefore, if $L_{i}$ is unbounded on $[a,b]$, then $L_{i}^2(a)$ is unbounded and $L_{i}^2\ge L_{i}^2(a)$ on $[0,a]$, or  
$L_{i}^2(b)$ is unbounded and $L_{i}^2\ge L_{i}^2(b)$ on $[b,1]$. In either case, we will eventually get a contradiction with \eqref{DCRST}. 

Now suppose that $\xi$ is unbounded. Since $\xi'\le 0$, 
we can find a subsequence such that $\xi(a)\to \infty$ or $\xi(b)\to -\infty$ monotonically in $k$. 
If $\xi(b)\to -\infty$, then since $\xi\le \xi(b)$ on $[b,1]$, 
$(L_{i}^2)'\ge 2L_{i}^2(- \xi(b))$ on $[b,1]$, and in order for \eqref{DCRST} to be satisfied, $L_{i}(b)$ must converge to $0$ for each $i$. 
Similarly, if $\xi(a)\to \infty$, then $L_{i}(a)\to 0$ 
for each $i$. In each case, we claim that far enough along the sequence, $\xi$ does not change sign on $[0,1]$. 
If it did, $L_{i}^2$ would be minimised where $\xi$ is $0$, so at this point, $L_{i}^2\le \min \{L_{i}(a)^2,L_{i}(b)^2\}\to 0$. 
We therefore have points $t^{(k)}\in [0,1]$ such that $(L^{(k)},\xi^{(k)})(t^{(k)})$ is getting arbitrarily close to the $(0,0)$ critical point of \eqref{RSET}, which contradicts the assumption that $(L^{(k)},\xi^{(k)})$ is unbounded on $[0,1]$.  
Therefore, $\xi$ does not change sign 
far enough along the sequence. However, since $\xi$ does not change sign, $\xi(t)$ is monotone decreasing in $t$ and we know that 
$\xi(a)\to \infty$ or $\xi(b)\to -\infty$, we have a contradiction 
with \eqref{DCRST}. 
\end{proof}
We now use Lemma \ref{CBUS} to demonstrate that if we have a sequence of solutions to \eqref{RSET} and \eqref{DCRST}, then these solutions are bounded. 
\begin{lemma}\label{APEDC}
Any solution of \eqref{RSET} and \eqref{DCRST} is bounded in the $C^0$ sense, independently of $p\in [0,1]$. 
\end{lemma}
\begin{proof}
Assume to the contrary that no such bound exists. 
Then there exists a sequence of solutions $(L^{(k)},\xi^{(k)},p^{(k)})$ which are unbounded in the $C^0$ sense, and we can assume that the 
$C^0$ norm of $(L^{(k)},\xi^{(k)})$ 
is monotone increasing. Similarly to the proof of Lemma \ref{CBUS}, we now drop reference to the superscript where convenient to do so. 

In order to keep track of the blow-up, we will let 
$M=\sqrt{tr(L^2)+\xi^2}$, and note that there exists $s>0$ such that \begin{equation}\label{EFR}
                                                                        \left|M'\right|\le sM^2. 
                                                                       \end{equation}                                                                    
Now take a subsequence such that either $\xi$ does not change sign for each element of the sequence, or $\xi$ does change sign for each element of the sequence. 
If $\xi$ does not change sign, then $\int_{0}^{1}M\le \left|c\right|+\sum_{i=1}^{d}\left|D_{i}\right|$. However, since $M$ is unbounded and 
satisfies \eqref{EFR}, we find that $\int_{0}^{1}M$ is also unbounded, which is a contradiction. 

We can now assume that $\xi$ changes sign for each element of our sequence.                                                                   
Let $(\bar{L},\bar{\xi})$ be the solution found by Lemma \ref{CBUS}. We claim that this solution is unbounded. Indeed, if there were 
some $R'$ which bounds $(\bar{L},\bar{\xi})$, choose $a$ and 
$b$ close to $0$ and $1$ respectively so that a solution of \eqref{RSET} bounded by 
$R'+1$ on $[a,b]$ stays bounded by $R'+2$ on $[0,1]$. Now by Lemma \ref{CBUS}, we have a subsequence which is 
convergent 
to $(\bar{L},\bar{\xi})$ in the $C^0[a,b]$ sense. 
Therefore, far enough along the subsequence, $(L,
\xi)$ will be bounded by $R'+1$ on $[a,b]$, and so must also be bounded by $R'+2$ on $[0,1]$. 
This contradicts the assumption 
that our sequence is unbounded in $C^0$. 

We now know that $(\bar{L},\bar{\xi})$ is unbounded. The remainder of the proof involves showing that the nature of the unboundedness implies that far enough 
along our sequence, $L^{(k)}$ and $\xi^{(k)}$ cannot possibly satisfy \eqref{DCRST}. 

$\textbf{Case}$ $\textbf{1}$: our solution $(\bar{L},\bar{\xi})$ is unbounded around $0$. 
Letting $\bar{M}^2=\bar{\xi}^2+tr(\bar{L}^2)$, we have again $\bar{M}'\ge -s\bar{M}^2$. 
Since $\bar{\xi}$ is monotone decreasing and $\bar{L}_i$ is monotone on any interval where $\bar{\xi}$ does not change sign, we see that
$\bar{M}$ blows up at $0$, so we must have $\bar{M}(t)\ge \frac{1}{st}$. 
Also note that for small $t$, $\bar{\xi}$ has to become positive, otherwise neither $\bar{\xi}$ nor $\bar{L}_i$ can become unbounded around $0$. 
Therefore, for small $t$, 
\begin{equation}\label{EFSTR}
 \bar{M}\le \sum_{i=1}^{d}\left|\bar{L}_{i}\right|+\bar{\xi}.
\end{equation}
Choose $i$ such that $\left|\bar{L}_{i}(t)\right|$ is maximised for all $t\in [0,1]$. 
Such an $i$ must exist 
and is independent of $t\in [0,1]$ because of the second equation of \eqref{RSET}.

$\textbf{Case}$ $\textbf{1}\textbf{a}$: $\left|\bar{L}_{i}\right|=0$, so $\bar{L}=0$ on $[0,1]$. 
Since $\xi$ changes sign, $tr(L^2)$ is minimised at the point where $\xi=0$. 
Now fix some $0<a<\frac{1}{2}<b<1$, and take a subsequence of $(L,\xi)$ that converges uniformly to 
$(0,\bar{\xi})$ on $[a,b]$. Since $L$ converges uniformly to $0$ on $[a,b]$, the minimum value of 
$tr(L^2)$ on $[0,1]$ must also converge to $0$. Then on this subsequence, we have points $t^{(k)}$ 
such that $\xi^{(k)}(t^{(k)})=0$ and $tr((L^{(k)})^2(t^{(k)}))\to 0$, so $M^{(k)}(t^{(k)})\to 0$. The inequality \eqref{EFR} then implies that $M^{(k)}\to 0$ uniformly on $[0,1]$, 
which is a contradiction because $M$ is unbounded.

$\textbf{Case}$ $\textbf{1}\textbf{b}$: $\left|\bar{L}_{i}\right|=\bar{L}_{i}>0$. In this case, \eqref{EFSTR} implies that $\bar{L}_{i}'=-\bar{\xi} \bar{L}_{i}\le-\bar{L}_i(\frac{1}{st}-d\bar{L}_{i})$. 
 The solution of the differential equation 
$x'=-x(\frac{1}{st}-dx)$ is $x(t)=\frac{s}{t(C-sd \ln(\frac{t}{s}))}$ for some $C\in \mathbb{R}$. If we choose $C$ so that $x(\frac{1}{2})=\bar{L}_{i}(\frac{1}{2})$, 
then for all 
$t\le \frac{1}{2}$, $\bar{L}_{i}(t)\ge x(t)$. Therefore $\int_{0}^{1}\bar{L}_i\ge \int_{0}^{\frac{1}{2}}\bar{L}_i =\infty$, which implies that $\int_{a}^{b}\bar{L}_i$ becomes arbitrarily large
by having $a$ and $b$ close to $0$ and $1$ respectively. 
However, on $[a,b]$, we have a subsequence of $(L,\xi)$ converging to 
$(\bar{L},\bar{\xi})$ in $C^0[a,b]$, and by taking another subsequence, we can assume that $L_i>0$ for each element of the sequence. 
We find that $$pD_{i}=\int_{0}^{1}L_{i}\ge \int_{a}^{b}L_{i}\to \int_{a}^{b}\bar{L}_{i},$$ which is a contradiction. 

$\textbf{Case}$ $\textbf{1}\textbf{c}$: $\left|\bar{L}_{i}\right|=-\bar{L}_{i}>0$. 
This time, \eqref{EFSTR} implies that $\bar{L}_{i}'\ge-\bar{L}_i(\frac{1}{st}+d\bar{L}_{i})$. Letting $L^*=-\bar{L}_i$, we find that $L^*>0$ and $(L^*)'\le -L^*(\frac{1}{st}-dL^*)$. This is the same inequality as in Case 1 b, so we proceed as before to find a contradiction.  

$\textbf{Case}$ $\textbf{2}$: $(\bar{L},\bar{\xi})$ is unbounded at $1$. We do not need to treat this case because the transformation 
$L\to -L$, $\xi\to -\xi$ and $t\to 1-t$ leaves \eqref{RSET} invariant, but changes the signs of $c$ and $D_{i}$, 
and shifts the location of the unboundedness from $1$ to $0$. 
The situation is then treatable by Case 1. 
\end{proof}
We now turn to techniques in Schauder degree theory. This theory concerns solutions of the equation $x-k(x)=z$, where $x$ is the unknown in a Banach space $X$, 
$z\in X$ is fixed, and $k:X\to X$ is a completely continuous map. If we focus on a bounded convex open subset $\Omega\subset X$, the Schauder degree is 
an integer $deg(I-k,\Omega,z)$ describing solutions of $x-k(x)=z$ for $x\in \Omega$, and can be defined if $x-k(x)\neq z$ for all $x\in \partial \Omega$. The following results are standard in the field of Schauder degree theory. 
These results can be found in \cite{ASDT}, for instance. 
\begin{theorem}\label{SDTT} 
The Schauder degree has the following properties: 

(i) If $deg(I-k,\Omega,0)\neq 0$, there exists $x\in \Omega$ such that $k(x)=x$.

(ii) If $(I-k')(a)$ is an invertible linear map for each $a$ with $a-k(a)=0$, then $deg(I-k,\Omega,0)=\sum _{a\in (I-k)^{-1}(0)} (-1)^{\sigma (a)}$, 
where $\sigma(a)$ is the sum of the algebraic multiplicities of the eigenvalues of $f'(a)$ contained in $(1,\infty)$. 

(iii) If $H:[0,1]\times \bar{\Omega}\to X$ is a completely continuous map such that $H(t,x)\neq x$ for all 
$x\in \partial \Omega$ and $t\in [0,1]$, then 
$deg(I-H(t,\cdot),\Omega,0)$ is independent of $t\in [0,1]$. 
\end{theorem}
With these results about the Schauder degree and Lemma \ref{APEDC}, we can prove the following.
\begin{lemma}\label{SD1P1}
 There exists a bounded open convex subset $\Omega\subset C^1([0,1]; \mathbb{R}^n)\times C^0([0,1];\mathbb{R})$ such that 
 $deg(I-H(1,0,\cdot),\Omega,0)=1$. 
\end{lemma}
\begin{proof}
 Lemma \ref{APEDC} implies that solutions of $(y,\xi)-H(p,0,y,\xi)=0$ are bounded in the $C^1([0,1]; \mathbb{R}^n)\times C^0([0,1];\mathbb{R})$ sense independently of $p\in [0,1]$. Therefore, there exists a
 bounded open convex subset $\Omega\subset C^1([0,1]; \mathbb{R}^n)\times C^0([0,1];\mathbb{R})$ containing all solutions of $(y,\xi)-H(p,0,y,\xi)=0$. 
 Property (iii) of Theorem \ref{SDTT} implies that $deg(I-H(1,0,\cdot),\Omega,0)=deg(I-H(0,0,\cdot),\Omega,0)$. 
 
 Now solutions $(y,\xi)$ of $(y,\xi)-H(0,0,y,\xi)=0$ correspond exactly to those pairs $(L,\xi)$ solving \eqref{RSET} with the condition that $\int_{0}^{1}L_i=\int_{0}^{1}\xi=0$. 
 It is straightforward to demonstrate that $L_i=\xi=0$ is the unique 
 solution. Furthermore, the linearisation of $H$ at this solution is trivial, so $deg(I-H(0,0,\cdot),\Omega,0)=1$ by property (ii) of Theorem \ref{SDTT}. 
\end{proof}
Since $deg(I-H(1,0,\cdot),\Omega,0)=1\neq 0$, there exists a pair $(y,\xi)$ solving $(y,\xi)-H(1,0,y,\xi)=0$, which is a solution of \eqref{QEEWP} with $p=1$ and $h^2=0$. 
A consequence of Lemma \ref{SD1P1} is the following. 
\begin{lemma}\label{TBOQE}
 Fix any $\Omega$ satisfying the conclusion of Lemma \ref{SD1P1}. There exists a $K>0$ such that for all $h^2<K$, $deg(I-H(1,h^2,\cdot),\Omega,0)=1$.
\end{lemma}
\begin{proof}
Set $p=1$. If no solution of \eqref{QEECOS} subject to \eqref{QEEWP} exists on $\partial \Omega$ for any $h^2>0$, then property (iii) of Theorem \ref{SDTT} implies that 
$deg(I-H(1,h^2,\cdot),\Omega,0)=1$ for all $h^2>0$. We can then take $K=\infty$. 

If on the other hand there is a $h^2$ for which a solution of \eqref{QEECOS} subject to \eqref{QEEWP} exists on $\partial \Omega$, then let $K$ be the infimum of all such numbers. 

$\textbf{Case}$ $\textbf{1}$: $K=0$. In this case, we have a sequence $((h^2)^{(k)},y^{(k)},\xi^{(k)})$ solving $(y^{(k)},\xi^{(k)})=H(1,(h^2)^{(k)},y^{(k)},\xi^{(k)})$ with $(h^2)^{(k)}\to 0$. 
 Since $H$ is completely continuous, 
 we know there exists a subsequence convergent to a solution of $(y,\xi)-H(1,0,y,\xi)=0$ in $\partial \Omega$, which contradicts the definition of $\Omega$. 
 
$\textbf{Case}$ $\textbf{2}$: $K>0$. By the definition of $K$, $\partial \Omega$ has no solution of \eqref{QEECOS} subject to \eqref{QEEWP} for any $h^2<K$. 
 Property (iii) of Theorem \ref{SDTT} then implies that $deg(I-H(1,h^2,\cdot),\Omega,0)=1$, for all $h^2<K$. 
\end{proof}
Lemma \ref{TBOQE} implies that we have a pair $(y,\xi)$ solving $(y,\xi)-H(1,h^2,y,\xi)=0$ for $h^2<K$. 
This pair then solves \eqref{QEECOS} subject to the integral conditions \eqref{QEEWP}, and Theorem~\ref{DCT} follows. 

\section{Examples of Non-existence and Non-uniqueness for the Dirichlet problem}
In the last section, we proved Theorem \ref{DCT}, establishing that cohomogeneity one quasi-Einstein metrics 
exists subject to Dirichlet conditions for $g$ and $u$. We showed that we could prescribe the value of $h^2$ in the interval $(0,K)$ for some $K>0$. 
In this section, we set $h^2=1$, and examine this same problem for two specific examples of homogeneous spaces $G/H$. 
The first example exhibits non-existence, demonstrating that the $K$ of Theorem \ref{DCT} cannot be arbitrarily large. The second example exhibits non-uniqueness, 
implying that, although uniqueness is not treated in 
Theorem \ref{DCT}, we cannot expect uniqueness without $h^2$ small. 

\subsection{Non-existence}
The following proposition demonstrates that for certain values of $\lambda$, 
we cannot expect solutions of the cohomogeneity one quasi-Einstein Dirichlet problem to exist if $h$ is too large, even if the homogeneous space is an abelian Lie group. 
\begin{prop}
 Suppose $G=\mathbb{S}^1$, $H$ consists of only the identity, $h=1$ and $\lambda >\pi^2$. 
 Then there is no $G$-invariant quasi-Einstein metric with constant $\lambda$ on $G/H\times (0,1)$ if we require that $u(0)=u(1)$. 
\end{prop}
\begin{proof}
If our homogeneous space is $\mathbb{S}^1$, then $n=d_1=d=1$ and $\beta_i=\gamma_{ik}^{l}=0$, and \eqref{QEECO} becomes 
\begin{align*}
L'+L^2-u''+mu'^2+\lambda&=0, \\
L'+L^2-u'L+\lambda&=0.
\end{align*} 
These two equations imply that 
\begin{align}\label{CTEMP}
u''=u'L+mu'^2.
\end{align} 
Since $u(0)=u(1)$, there must be a $t\in [0,1]$ with $u'(t)=0$, so \eqref{CTEMP} implies that $u'=0$ on $[0,1]$, and $L$ solves the equation
$L'+L^2+\lambda=0$ on $[0,1]$. No continuous solution exists if $\lambda>\pi^2$. 
\end{proof}

\subsection{Non-uniqueness}
The following theorem demonstrates that we can also achieve non-uniqueness of solutions to the cohomogeneity one Dirichlet problem for quasi-Einstein metrics. 
\begin{theorem}\label{NUDP}
 Suppose $G=SO(3)$, $H=SO(2)$, $\lambda=m=0$ and $h=1$. 
 Then for some choice of $\hat{g}$, there exists at least two 
 cohomogeneity one quasi-Einstein metrics on $G/H\times (0,1)$ satisfying the Dirichlet conditions $u(0)=0=u(1)$ and $g_0=g_1=\hat{g}$. 
\end{theorem}
\begin{remark}
Since we are assuming $m=0$, the quasi-Einstein metric we find is a Ricci soliton. 
\end{remark}
The homogeneous space $G/H$ is isotropy irreducible, so \eqref{QEECOS} becomes 
\begin{align}\label{RSES2V}
 \begin{split}
 \xi'&=-2(y')^2,\\
 y''&=\beta e^{-2y}-\xi y',
 \end{split}
\end{align}
where $\beta>0$ depends on our choice of reference metric on $G/H=\mathbb{S}^2$. By scaling this reference metric, we can take $\beta=1$. 

To prove Theorem \ref{NUDP}, it suffices to prove that there are multiple solutions of \eqref{RSES2V} satisfying 
\begin{align*}
 y(0)=y(1)&=\bar{y},\\
 \int_{0}^{1}\xi&=0,
\end{align*}
for some choice of $\bar{y}\in \mathbb{R}$.
To find solutions, we will use the following two lemmas.
\begin{lemma}\label{SOS2}
 Suppose that we have a solution of \eqref{RSES2V} on $[0,1]$ such that
  $\xi(\frac{1}{2})=y'(\frac{1}{2})=0$. 
 Then $y(0)=y(1)$ and $\int_{0}^{1}\xi=0$. 
\end{lemma}
\begin{proof}
 Consider the functions $\tilde{y}(t)=y(1-t)$ and $\tilde{\xi}(t)=-\xi(1-t)$. We have 
 $\tilde{\xi}(\frac{1}{2})=\tilde{y}'(\frac{1}{2})=0$, $\tilde{y}(\frac{1}{2})=y(\frac{1}{2})$, and the pair $(\tilde{\xi},\tilde{y})$ also satisfies 
 \eqref{RSES2V}. Therefore, $(\tilde{\xi},\tilde{y})$ must be identical to $(\xi,y)$ on $[0,1]$. In particular, $y(0)=\tilde{y}(0)=y(1)$, and 
 $\xi(t)=-\xi(1-t)$, so $\int_{0}^{1}\xi=0$. 
\end{proof}
\begin{lemma}\label{SEFKL}
There exists $k\in \mathbb{R}$ such that a solution of \eqref{RSES2V} subject to the conditions
\begin{align}\label{DCES}
 \begin{split}
  \xi(\tfrac{1}{2})=0,\qquad
  y'(\tfrac{1}{2})=0,\qquad
  y(\tfrac{1}{2})=k_1,
 \end{split}
\end{align} 
exists on $[0,1]$ whenever $k_1>k$.
\end{lemma}
\begin{proof}
Consider the complete metric space 
$$X=\{(\xi,y,z)\in C^0([\tfrac{1}{2},1];\mathbb{R}^3): \left|\left|\xi \right|\right|_{C^0[\frac{1}{2},1]}\le R,
\left|\left|y-k_1 \right|\right|_{C^0[\frac{1}{2},1]}\le 1, \left|\left|z \right|\right|_{C^0[\frac{1}{2},1]}\le R\}$$
for some $R>0$. The problem of solving \eqref{RSES2V} subject to \eqref{DCES} can alternately be formulated as finding a fixed point of $H:C^0([\frac{1}{2},1];\mathbb{R}^3)\to C^0([\frac{1}{2},1];\mathbb{R}^3)$, where the first, second and third components of $H(\xi,y,z)$ are given by 
\begin{align*}
&\int_{\frac{1}{2}}^{t}-2z(s)^2\,ds,\\
&k_1+\int_{\frac{1}{2}}^{t}z(s)\,ds,\\
&\int_{\frac{1}{2}}^{t}e^{-2y(s)}-\xi(s)z(s)\,ds, 
\end{align*}
respectively. For large $k_1$ and small $R$, $H$ is a contraction on $X$. 
The result then follows from the Banach fixed point Theorem.
\end{proof}
By Lemma \ref{SOS2}, we know that to find a solution of \eqref{RSES2V} with 
$\int_{0}^{1}\xi=0$, and $y(0)=y(1)=\bar{y}$, it suffices to find a solution of \eqref{RSES2V} with \eqref{DCES}, where $k_1$ is chosen so that $y(1)=\bar{y}$. 
The following lemma demonstrates that a choice of $k_1$ is not always unique, completing the proof of Theorem \ref{NUDP}.  
\begin{lemma}
 There exists values of $\bar{y}$ such that there are at least two values of $k_1$ for which a solution of \eqref{RSES2V} and \eqref{DCES} 
  satisfies $y(1)=\bar{y}$. 
\end{lemma}
\begin{proof}
Let $k^*$ be the infimum of all values of $k$ such that a solution of \eqref{RSES2V} with 
\eqref{DCES} exists on $[0,1]$ whenever $k_1>k$. Such a $k^*$ exists because of Lemma 
\ref{SEFKL}. Then 
for all $k_1\in (k^*,\infty)$, there exists a solution of \eqref{RSES2V} with 
\eqref{DCES} on $[0,1]$. Since the value of $y(1)$ depends continuously on $k_1$, the proof will be complete if we can demonstrate that $y(1)$ does not 
depend monotonically on $k_1\in (k^*,\infty)$. We do this by ruling out certain cases.

$\textbf{Case}$ $\textbf{1}$: $y(1)$ is decreasing as $k_1$ increases on $(k^*,\infty)$. 
However, by taking $k_1\to \infty$, we see that $y(1)\to \infty$ as well because $y(1)\ge k_1$, which contradicts the assertion that $y(1)$ is monotone decreasing. 

$\textbf{Case}$ $\textbf{2}$: $y(1)$ is increasing as $k_1$ increases and $k^*=-\infty$. Now $(2y'^2-\xi^2)'=4y'(e^{-2y}-\xi y')+4\xi(y')^2=4y'e^{-2y}$, so 
from the fact that $2y'^2-\xi^2=0$ at $t=\frac{1}{2}$ and $y''\ge 0$ on $[0,1]$, we conclude that 
\begin{align}\label{IIS}
2 y'^2(t)\ge \xi(t)^2
\end{align}
for all $t\in [0,1]$. We 
then find that \begin{align}\label{EFVI}
                              \xi'\le -(\xi)^2.
                             \end{align}
 This implies that as $k_1\to -\infty$, $\xi\left(\frac{7}{8}\right)$ remains bounded from below, otherwise \eqref{EFVI} implies that $\xi$ blows up before $t$ gets to $1$. 
 Now using the fact that $y''\ge 0$ on $[0,1]$, we see that 
 \begin{align*}
  \xi(\tfrac{7}{8})&=-\int_{\frac{1}{2}}^{\frac{7}{8}}2(y'(s)^2)ds\\
  &\le -\int_{\frac{3}{4}}^{\frac{7}{8}}2(y'(s)^2)ds\\
  &\le -\frac{1}{4}y'(\tfrac{3}{4})^2,
 \end{align*}
so $y'(\frac{3}{4})^2$ is also bounded as $k_1\to -\infty$. Again using the inequality $y''(t)\ge 0$, 
we find that $y'$ is bounded on $[\frac{1}{2},\frac{3}{4}]$. This implies that 
$y-k_1$ is bounded on $[\frac{1}{2},\frac{3}{4}]$, from which we find that as $k_1\to -\infty$, $e^{-2y}$ is getting arbitrarily large on $[\frac{1}{2},\frac{3}{4}]$, whence the second equation of \eqref{RSES2V} 
implies that $y'(\frac{3}{4})$ is getting arbitrarily large, a contradiction. 

$\textbf{Case}$ $\textbf{3}$: $y(1)$ is increasing as $k_1$ increases and $k^*>-\infty$. In this case, we claim that a solution of \eqref{RSES2V} subject to \eqref{DCES} 
exists for $k_1=k^*$. To see this, take a sequence of $k_1>k^*$ such that $k_1\to k^*$. By our assumption on $y(1)$ 
and the monotonicity of $y$, we know that $\left| y(t)\right|$
is bounded 
on $[0,1]$, say by $R>0$. Therefore, \eqref{IIS} implies that $0\le y''\le e^{2R}+\sqrt{2}(y')^2=\phi(\left|y'\right|)$, where 
$\phi:[0,\infty)$ satisfies $\int_{0}^{\infty}\frac{s}{\phi(s)}ds=\infty$. 
Lemma 5.1 in Chapter 12 of \cite{Hartman} then implies that there exists an $M>0$ such that 
$\left|y'\right|<M$ on $[0,1]$. These estimates on $y$ and $y'$ alongside \eqref{RSES2V} imply 
that $y$ is bounded in $C^3[0,1]$ and $\xi$ is bounded in $C^2[0,1]$. 
The Arzela-Ascoli Theorem then implies that we have a convergent subsequence of 
$(y,\xi)$ in $C^2[0,1]\times C^1[0,1]$. 
The limit is clearly going to be a solution of \eqref{RSES2V} subject to \eqref{DCES} with $k_1=k^*$. 

Since a solution of \eqref{RSES2V} with 
\eqref{DCES} exists for $k_1=k^*$, we can use basic perturbation arguments to prove existence of solutions to \eqref{RSES2V} and \eqref{DCES} for some values of $k_1<k^*$. This contradicts the definition of $k^*$. 
\end{proof}

\section{Acknowledgements}
I am grateful to my advisory team Artem Pulemotov, J{\o}rgen Rasmussen and Joseph Grotowski for their supervision of the project. 
I would also like to thank Andrew Dancer for suggesting the problem to me. Finally, I am very grateful to Matthias Wink for numerous discussions on this project as well as 
many other topics in mathematics. 
This research was supported by the Australian Government through a Research Training Program scholarship as well as through Artem Pulemotov’s Discovery Early-Career Researcher Award DE150101548.

\end{document}